\documentclass[11pt]{article}

\usepackage{amsthm, amsmath, amssymb, amsfonts, url, booktabs, tikz, setspace, fancyhdr, bm}
\usepackage[margin=0.9in]{geometry}
\usepackage{hyperref, enumerate}
\usepackage[shortlabels]{enumitem}
\usepackage[babel]{microtype}
\usepackage[english]{babel}
\usepackage[capitalise]{cleveref}
\usepackage{comment}
\usepackage{bbm}
\usepackage{csquotes}
\usepackage{graphicx}
\usepackage{float}
\usepackage[dvipsnames]{xcolor}
\usepackage{soul}
\usepackage{mathtools}
\usepackage[normalem]{ulem}
\usepackage{titlesec}

\usetikzlibrary{
    positioning,
    arrows.meta,
    shapes.geometric,
    decorations.pathmorphing,
    decorations.pathreplacing,
    fit,
    backgrounds
}

\titlespacing*{\section}
  {0pt}{2.5ex plus .5ex minus .2ex}{1.2ex plus .2ex}

\titlespacing*{\subsection}
  {0pt}{2ex plus .4ex minus .2ex}{0.8ex plus .2ex}

\counterwithin{figure}{section}
\numberwithin{equation}{section}

\newtheorem{theorem}{Theorem}[section]

\newtheorem{conjecture}[theorem]{Conjecture}
\newtheorem{question}[theorem]{Question}
\newtheorem{lemma}[theorem]{Lemma}

\newtheorem{claim}[theorem]{Claim}

\theoremstyle{definition}
\newtheorem{defn}[theorem]{Definition}
\newtheorem*{defn-non}{Definition}

\newlist{Case}{enumerate}{2}
\setlist[Case,1]{%
    label={\bfseries Case \arabic*.},
    labelindent=1em,
    labelwidth=1.3cm,
    labelsep*=1em,
    leftmargin=!
}
\setlist[Case,2]{%
    label={\bfseries Subcase \arabic{Casei}.\arabic*.},
    labelindent=-1em,
    labelwidth=1.3cm,
    labelsep*=1em,
    leftmargin=!
}

\newcommand{\ceil}[1]{\lceil #1\rceil}

\newcommand{\E}{\mathbb E}
\newcommand{\PP}{\mathbb P}

\newcommand{\Rind}{\widehat R_{\mathrm{ind}}}

\crefname{theorem}{Theorem}{Theorems}
\Crefname{theorem}{Theorem}{Theorems}
\crefname{lemma}{Lemma}{Lemmas}
\Crefname{lemma}{Lemma}{Lemmas}
\crefname{conjecture}{Conjecture}{Conjectures}
\Crefname{conjecture}{Conjecture}{Conjectures}
\crefname{corollary}{Corollary}{Corollaries}
\Crefname{corollary}{Corollary}{Corollaries}
\crefname{proposition}{Proposition}{Propositions}
\Crefname{proposition}{Proposition}{Propositions}
\crefname{remark}{Remark}{Remarks}
\Crefname{remark}{Remark}{Remarks}

\hypersetup{
  pdftitle={Spread Methods for Induced Cycles},
  pdfsubject={Induced cycles, induced size--Ramsey numbers, and random regular graphs},
  pdfkeywords={induced cycle, induced size--Ramsey number, odd cycle, hypergraph containers, random regular graph}
}

\title{Induced Size–Ramsey Numbers of Odd Cycles via Spread Measures}

\author{
Lanchao Wang\thanks{School of Mathematics, Nanjing University, Nanjing, China, and ECOPRO, Institute for Basic Science, Daejeon, Korea. Email: lanchaowang@foxmail.com.}
\and
Xiaolin Wang\thanks{School of Mathematics and Statistics, Fuzhou University, Fuzhou, China. Email: xiaolinw@fzu.edu.cn.}
}

\date{}

\begin{document}

\maketitle\begin{abstract}
We prove that, for every $k\ge2$ and every sufficiently large odd $n$,
the induced size--Ramsey number of $C_n$ satisfies
\[
\widehat R_{\mathrm{ind}}(C_n;k)=e^{\Theta(k)}n.
\]
This improves the previous upper bound of $e^{O(k\log k)}n$
and resolves a conjecture of Brada\v{c}, Dragani\'c and Sudakov.
We obtain this result by proving a stronger form of their odd-hole
gadget conjecture: we construct a graph with $e^{O(k)}$ edges
whose every $k$-edge-colouring contains a monochromatic induced odd
cycle of length $O(\log k)$.

Our proof develops a spread-based approach to finding induced cycles.
The key idea is to place spread probability weights directly on
families of candidate cycles and control their possible chords
through a quadratic chord score. Combining this with random
sparsification and hypergraph containers yields a sparse
host graph that works for every $k$-edge-colouring.
\end{abstract}

\section{Introduction}

Induced cycles play a central role in graph theory. For example, the
celebrated Strong Perfect Graph Theorem, proved by Chudnovsky, Robertson,
Seymour, and Thomas~\cite{CRST}, states that a graph is perfect if and
only if neither the graph nor its complement contains an induced odd
cycle of length at least five. Induced cycles have also been widely
studied in extremal graph theory, particularly from the viewpoint of
counting and maximising them; see, for
example,~\cite{BHLP,HT,KNV,Lidicky,MS,PG}.

In this paper, we study the induced size--Ramsey numbers of odd cycles.
The $k$-colour size--Ramsey number of a graph $H$, denoted by
$\widehat R(H;k)$, is the minimum number of edges in a graph $G$ such that
every $k$-edge-colouring of $G$ contains a monochromatic copy of $H$.
A natural strengthening is obtained by requiring the monochromatic copy to
be induced in the host graph. The corresponding parameter, denoted by
$\widehat R_{\mathrm{ind}}(H;k)$, is the minimum number of edges in a graph
$G$ such that every $k$-edge-colouring of $G$ contains a monochromatic copy
of $H$ which is induced in $G$. Clearly, every induced monochromatic copy of $H$ is also a monochromatic copy, and hence
$
\widehat R(H;k)\le \widehat R_{\mathrm{ind}}(H;k).
$
However, a monochromatic copy of $H$ need not be induced, since its vertices may span additional edges of any colour between pairs that are nonadjacent in $H$. Controlling these extra edges is a central difficulty in the induced setting. Moreover, unlike in the ordinary setting, adding edges to the host graph may destroy an induced copy.

The finiteness of the induced Ramsey number was proved independently by
Deuber, Erd\H{o}s, Hajnal and Pósa, and R\"odl~\cite{Deuber,EHP,Rodl}.
Paths and cycles are among the most basic examples. For paths, a substantial
line of work has studied both the ordinary and induced size--Ramsey numbers;
see~\cite{Beck,DGK,DK,DP,DPnote,GH,HKL,Krivelevich}. In particular,
Beke, Li and Sahasrabudhe~\cite{BLS} recently proved
$\widehat R(P_n;k)=\Theta(k^2\log k)n$ whenever $n\ge200\log k$.

The ordinary size--Ramsey numbers of cycles are also well understood and,
somewhat surprisingly, exhibit rather different behaviour according to the
parity. For even cycles, a sequence of works
\cite{BDS,JKM,JM,WangWangEvenCycles} led to the sharp bound
$\widehat R(C_n;k)=\Theta(k^2\log k)n$ whenever $n\ge200\log k$. For odd cycles, the sharp bound
$\widehat R(C_n;k)=e^{\Theta(k)}n$ was established
in~\cite{BDS,JM} for sufficiently long odd $n$.

We now turn to the induced setting. Haxell, Kohayakawa and
{\L}uczak~\cite{HKL} proved
$\widehat R_{\mathrm{ind}}(C_n;k)=\Theta_k(n)$ for every fixed $k$.
Their argument relies on the regularity lemma and therefore yields a
tower-type dependence on the number of colours.
The dependence on $k$ again differs sharply according to the
parity.
For even cycles, Brada\v{c}, Dragani\'c and Sudakov~\cite{BDS} obtained
$
\widehat R_{\mathrm{ind}}(C_n;k)=O(k^{102})n.
$
For odd cycles, they proved
$
\widehat R_{\mathrm{ind}}(C_n;k)=e^{O(k\log k)}n.
$
This raises a natural question: does enforcing inducedness really
require the additional $\log k$ factor in the exponent, or is this
a limitation of the existing construction?
They conjectured that the upper bound could be improved to
$
\widehat R_{\mathrm{ind}}(C_n;k)=e^{O(k)}n
$
for sufficiently large odd $n$. This would be best possible up to the
constant in the exponent, since
$\widehat R(C_n;k)=e^{\Theta(k)}n$ for sufficiently long odd cycles
$C_n$~\cite{BDS,JM}.

Our main result  determines the optimal
exponential dependence.

\begin{theorem}\label{thm:long-cycles}
For every $k\ge2$, there exists $n_0=n_0(k)$ such that every odd
$n\ge n_0$ satisfies
\[
\Rind(C_n;k)=e^{\Theta(k)}n.
\]
\end{theorem}

To approach this problem, Brada\v{c}, Dragani\'c and Sudakov~\cite{BDS}
formulated the following odd-hole gadget conjecture, which they
observed would imply the desired bound
$\widehat R_{\mathrm{ind}}(C_n;k)\le e^{O(k)}n$.
Their reduction exploits the fact that an induced odd cycle
contains two induced paths with the same endpoints whose lengths
differ by one. By using expansion to combine many such gadgets
and choosing between these two paths, they obtain monochromatic
induced cycles of any sufficiently large prescribed odd length.

\begin{conjecture}[\cite{BDS}]\label{conj:BDS-gadget}
For every integer $k$, there is a graph $G$ with $e^{O(k)}$ edges which,
for any $k$-colouring of its edges, contains a monochromatic odd cycle of
length at least $5$ as an induced subgraph.
\end{conjecture}

We resolve this conjecture in a stronger form: the monochromatic
induced odd cycle can always be chosen to have logarithmic length.

\begin{theorem}\label{thm:gadget}
For every integer $k\ge3$, there is a graph $F_k$ with
$e(F_k)\le e^{O(k)}$ whose every $k$-edge-colouring contains a
monochromatic induced cycle $C_h$ for some odd integer
$
5\le h\le4\ceil{\log_2 k}.
$
\end{theorem}

Brada\v{c}, Dragani\'c and Sudakov~\cite{BDS} previously constructed
an $e^{O(k\log k)}$-edge graph whose every $k$-edge-colouring contains
a monochromatic induced $C_5$, using a pseudorandom triangle-free graph.
While every $C_5$ in a triangle-free graph is automatically induced,
this is no longer true for longer odd cycles. Moreover, a recent
result of Steiner~\cite{Steiner} shows that the multicolour Ramsey
number of odd cycles of bounded length is superexponential in $k$.
Consequently, a graph with $e^{O(k)}$ edges cannot force a
monochromatic odd cycle of bounded length. Thus, allowing the cycle
length to grow with $k$ is essential, and the main difficulty is to
ensure inducedness without excluding chords through structural
restrictions on the host graph.

Our proof of \cref{thm:gadget} develops a spread-based approach to
overcome this difficulty. The use of spread measures goes back to
Talagrand~\cite{Talagrand}. Roughly speaking, a probability weight
is spread if its mass is not too concentrated on any small prescribed
configuration. Spreadness has previously been used in connection with
induced cycles by D\v{z}avoronok, Gabsdil, Mylet, Popa and
Shao~\cite{DGMPSh}. Their approach uses spreadness to control how often
vertices are visited by the random walks from which the cycles are
constructed, and hence controls possible chords indirectly.

Rather than imposing structural restrictions on the host graph to
exclude chords or using spreadness to control chords indirectly
through random walks, we consider many candidate cycles simultaneously
and place a spread probability weight directly on them. This allows
us to control their possible chords through a quadratic chord score.
A sufficiently small chord score, relative to the number of edges
in the host graph, guarantees that one of the candidate cycles is
induced. By combining this with random sampling, we can obtain
such a weight while retaining enough candidate cycles.

The remaining difficulty is that a size--Ramsey host graph must
be chosen before its edges are coloured, whereas sparsification
for a fixed colouring does not guarantee a single graph that works
for all colourings. To overcome this, we use the hypergraph
container method, developed independently by Balogh, Morris and
Samotij~\cite{BMS} and Saxton and Thomason~\cite{ST},
to compress the edge sets that do not support a probability weight
with sufficiently small chord score. This allows us to construct
a single sparse graph that works for every $k$-edge-colouring.
We believe that our approach may also be useful for other problems
concerning induced cycles and induced structures.

\medskip
\noindent\textbf{Organization of the paper.}
In Section~\ref{sec:overview}, we give an overview of the proof of
Theorem~\ref{thm:gadget}. In Section~\ref{sec:spread}, we develop the
spread framework for controlling chords. In Section~\ref{sec:compression},
we use hypergraph containers to reduce the remaining possibilities to a
small family of containers. Finally, in Section~\ref{sec:gadget-proof},
we complete the proof of Theorem~\ref{thm:gadget}, which yields
Theorem~\ref{thm:long-cycles}.
 
\section{Proof overview  of
Theorem~\ref{thm:gadget}}
\label{sec:overview}
Throughout this overview, we assume that $k$ is
sufficiently large and set $B=16^k$, $M=B^{20}$
and $p=B^{-9}$.
\medskip

\noindent\textbf{Step 1: Spread weights and chord control.}
We begin with the complete graph $K_M$.
Given any $k$-edge-colouring, a simple averaging argument
yields, in some colour, a spread probability weight $\mu$
on short odd cycles of a fixed length $h=O(\log k)$.
Although none of these cycles is induced in $K_M$,
spreadness ensures that the weighted cycles do not
concentrate too heavily on any prescribed set of vertices.
Consequently, their possible chords are well distributed
over the edges of $K_M$, regardless of their colours,
allowing us to control them after sparsification.

\smallskip
\noindent\textbf{Chord score and inducedness.}
The key idea is to control these possible chords
simultaneously, rather than trying to eliminate them
from each candidate cycle individually.
For a cycle $Q$, let $\operatorname{Ch}(Q)$ denote the set
of pairs of nonconsecutive vertices of $Q$.
For each pair $xy$, let
\[
c_\mu(xy)=
\sum_{Q:\,xy\in\operatorname{Ch}(Q)}\mu(Q)
\]
be the probability that $xy$ is a possible chord of a
random cycle chosen according to $\mu$.
We introduce the quadratic chord score
$
\operatorname{ch}(\mu)=\sum_{xy}c_\mu(xy)^2,
$
which measures how concentrated the possible chords are.
Its importance comes from a simple observation: if a graph
$G$ supports $\mu$, then the expected number of chords
of a random cycle is at most
\[
\sum_{xy\in E(G)}c_\mu(xy)
\le \sqrt{e(G)\operatorname{ch}(\mu)}.
\]
Consequently, if $e(G)\operatorname{ch}(\mu)<1$,
some positive-weight cycle must be induced in $G$.

Thus, finding an induced cycle reduces to constructing
a sparse graph supporting a probability
weight with small chord score. To achieve this, we seek to reduce the number of edges
while retaining enough weight on the candidate cycles
and keeping their chord score under control.

\smallskip

\noindent\textbf{Random sampling and conditional thinning.}
We sample each edge of $K_M$ independently with
probability $p$ and retain only those candidate cycles
whose edges are all sampled.
The main difficulty is that cycles sharing many edges
are more likely to survive together, potentially increasing
the chord score after renormalisation.
We overcome this by using spreadness to control the
overlaps between pairs of weighted cycles. These estimates
allow us to retain enough total weight while keeping the
normalised chord score sufficiently small.
In Step~2, we will also use this argument when the
sampled edge set is required to be non-chord-light.
For this, we need to ensure that the candidate cycles
still have sufficiently large probabilities of
surviving under this condition.

\medskip

\noindent\textbf{Step 2: Compression by containers.}
We call an edge set $I\subseteq E(K_M)$
\emph{chord-light} if, for some odd
$h$ with $5\le h\le4\ceil{\log_2 k}$,
it supports a probability weight $\mu$ on
$h$-cycles satisfying
$\operatorname{ch}(\mu)\le B^{-37}$.
The random sampling and conditional thinning
argument from Step~1 provides a way to obtain
such chord-light edge sets from spread weights.

By the chord score criterion, it suffices to construct
a graph with at most $2B^{36}$ edges in which every
$k$-edge-colouring has a chord-light colour class.

\smallskip

\noindent\textbf{Container construction and spread exclusion.}
Using two hypergraph container lemmas, we construct
a family $\mathfrak C\subseteq 2^{E(K_M)}$ with
$|\mathfrak C|\le\exp(B^{32})$ such that every
non-chord-light edge set is contained in some
member of $\mathfrak C$. 

The crucial additional
property is that none of these containers supports
a spread probability weight on short odd cycles.
To see why, suppose that a container supports such
a weight $\mu$. We consider random edge sampling
under a condition that forces the sampled edge set
to be non-chord-light.
The two container lemmas identify certain cycles
that may fail to survive this conditioning.
Since $\mu$ is spread, these cycles carry only a
small amount of weight and can be discarded.
The remaining cycles still have a sufficiently large
probability of surviving the conditioned sampling.
The conditional thinning lemma from Step~1 then
produces a chord-light edge set, contradicting
the condition imposed on the sampling.

\smallskip

\noindent\textbf{Robust non-covering.}
The spread-weight construction in Step~1 is robust under
the deletion of a small set of edges. Thus, if $k$ edge
sets cover almost all edges of $K_M$, at least one must
support a spread probability weight. Since none of our
containers does, we obtain
\[
\left|E(K_M)\setminus\bigcup_{i=1}^k X_i\right|
\ge B^{-3}\binom M2 \  \ \ \text{for every $X_1,\ldots,X_k\in\mathfrak C$.}
\]

\medskip
\noindent\textbf{Step 3: Construction of the host graph.}
We construct $F$ by greedily selecting edges that hit
every set of the form
$E(K_M)\setminus\bigcup_{i=1}^k X_i$,
where $X_1,\ldots,X_k\in\mathfrak C$.
By the preceding bound, each such set contains at least
a $B^{-3}$ fraction of the edges of $K_M$.
We can therefore choose each edge to hit at least
a $B^{-3}$ fraction of the remaining sets. Since there
are at most $|\mathfrak C|^k$ such sets, at most
$2B^{36}$ edges suffice.

Now consider any $k$-edge-colouring of $F$.
Some colour class must be chord-light; otherwise
the $k$ colour classes would lie in $k$ containers,
contradicting the construction of $F$.
Hence some colour class supports a probability weight
$\mu$ with $\operatorname{ch}(\mu)\le B^{-37}$.
Since
$
e(F)\operatorname{ch}(\mu)\le 2B^{-1}<1,
$
our chord score criterion guarantees a monochromatic
induced odd cycle, completing the proof of
\cref{thm:gadget}.

\section{Spread weights, chord control, and induced cycles}
\label{sec:spread}
The purpose of this section is to construct, in a complete graph on
$e^{O(k)}$ vertices, a well-distributed weighted family of short
monochromatic odd cycles. Such a family is robust under deleting a small
set of edges, has controlled overlaps, and has well-distributed possible
chords. These properties will later be used to force an induced
monochromatic cycle. 

For this section, assume that $k\ge3$, let $B=16^k$, and let $M\ge B$.
We keep $M$ general initially and will later specialise to $M=B^{20}$.

In \cref{sec:short-cycles}, we establish the existence of short
monochromatic odd cycles in $K_B$.
In \cref{sec:spread-weights}, we use them to construct spread probability
weights on short monochromatic odd cycles in $K_M$. In
\cref{sec:overlap}, we establish the overlap estimates needed later in the
container argument. In \cref{sec:chords}, we use these estimates to control possible
chords and show that a sufficiently small chord score forces an induced
cycle. Finally, in \cref{sec:conditional-thinning}, we show that a spread
weight can be thinned under decreasing conditioning while retaining control
of the chord score.

\subsection{Short monochromatic cycles in complete graphs}\label{sec:short-cycles}

We show that every complete graph of exponential order contains a short
monochromatic odd cycle under any $k$-edge-colouring.

\begin{lemma}\label{lem:short-cycle}
Every $k$-edge-colouring of $K_{16^k}$ contains a
monochromatic odd cycle $C_h$ for some
$
   5\le h\le 4\ceil{\log_2 k}.
$
\end{lemma}

The proof combines a simple breadth-first-search (BFS) observation with a colouring result of
Axenovich, Cames van Batenburg, Janzer, Michel and Rundstr\"om~\cite{ACJMR}.
If a monochromatic graph contains no short odd cycle, then each BFS level
inside a suitable ball is 3-colourable, and hence the whole ball is
6-colourable.  Their result then bounds the order of the ambient complete
graph, contradicting the choice of $K_{16^k}$.

The ball of radius $r$ around a vertex $v$ in $G$ is the subgraph induced by
all vertices at distance at most $r$ from $v$.  

\begin{lemma}[\cite{ACJMR}, Lemma~2.1]\label{lem:local-colouring}
 Consider a $k$-edge-colouring of $K_N$.  Suppose that, in each
monochromatic subgraph, the ball of radius $r$ around every vertex has
chromatic number at most $q$.  Then
$
   N\le q^k k^{k/r}.
$
\end{lemma}

To apply \cref{lem:local-colouring}, we verify that every monochromatic
subgraph satisfies the required local colouring condition.

\begin{lemma}\label{lem:six-colours}
Let $r\ge2$.  If a graph contains no odd cycle of length in
$\{5,7,\ldots,2r+1\}$, then the ball of radius $r$ around every vertex is
6-colourable.
\end{lemma}

\begin{proof}
Fix a vertex $v$ and consider the ball of radius $r$ around $v$.  Choose a
BFS tree $T$ rooted at $v$.  For $0\le i\le r$, let $L_i$
be the $i$-th level, consisting of the vertices at distance exactly $i$ from
$v$.  We first show that $G[L_i]$ is 3-colourable for every $i$.

We claim that the endpoints of every edge inside $L_i$ have the same parent
in $T$.  Indeed, let $xy\in E(G[L_i])$, and let the last common
vertex of the tree paths from $v$ to $x$ and $y$ lie in $L_j$.  If
$j\le i-2$, these two tree paths together with $xy$ form an odd cycle of
length
$
   2(i-j)+1\in\{5,7,\ldots,2r+1\},
$
a contradiction.  Hence $j=i-1$, so $x$ and $y$ have the same parent.

Fix $u\in L_{i-1}$ and consider the subgraph induced by its children in
$L_i$.  This subgraph contains no path on four vertices, since such a path together
with $u$ would form a $5$-cycle. Thus, $G[L_i\cap N_T(u)]$ is a disjoint union of stars and triangles, and
is therefore $3$-colourable.

Since every edge inside $L_i$ joins two children of the same parent,  there are
no edges between the child sets of distinct vertices in $L_{i-1}$.  It
follows that the whole level $L_i$ is 3-colourable.  Finally, use one set of
three colours on the even levels and a disjoint set of three colours on the
odd levels.  Since every edge of the ball has endpoints either in the same
level or in consecutive levels, this is a proper 6-colouring of the ball.
\end{proof}

\begin{proof}[Proof of \cref{lem:short-cycle}]
Note that
$2\ceil{\log_2 k}+1\le4\ceil{\log_2 k}$.
Suppose that there is no monochromatic odd cycle of length between
$5$ and $2\ceil{\log_2 k}+1$ in $K_{16^k}$. Applying \cref{lem:six-colours} on $r=\ceil{\log_2 k}$,
every ball of radius $\ceil{\log_2 k}$ in each monochromatic subgraph is
$6$-colourable. Applying \cref{lem:local-colouring} with
$(N,q,r)=(16^k,6,\ceil{\log_2 k})$, we obtain
$
16^k\le 6^k k^{k/\ceil{\log_2 k}}
\le 12^k,
$
where the last inequality follows from
$k^{1/\ceil{\log_2 k}}\le2$, a contradiction.
\end{proof}

\subsection{Spread weights on short monochromatic cycles}\label{sec:spread-weights}
We assign a spread probability weight to the short monochromatic odd cycles found above. 

For an odd $h$ with $5\le h\le4\ceil{\log_2 k}$, let
$\mathcal C_h$ denote the family of all $h$-cycles in $K_M$. Throughout the paper, all weights are assumed to be nonnegative.

\begin{defn}\label{def:spread}
A \emph{spread weight} on $\mathcal C_h$ is a  weight $\mu$
satisfying
$
\sum_{Q\in\mathcal C_h}\mu(Q)\le1
$
and, for every $S\subseteq V(K_M)$ with $2\le|S|\le h$, \vspace{-0.3cm}
\begin{equation}\label{eq:spread}
\sum_{Q\in\mathcal C_h:\,S\subseteq V(Q)}\mu(Q)
\le
8k\ceil{\log_2 k}\left(\frac BM\right)^{|S|}.
\end{equation}
A weight $\mu$ on $\mathcal C_h$, not necessarily spread, is a \emph{probability weight} if 
$\sum_{Q\in\mathcal C_h}\mu(Q)=1$.
\end{defn}
Thus, a spread weight does not concentrate on any fixed small vertex set. The use of spread measures goes back to Talagrand~\cite{Talagrand}; see also
D\v{z}avoronok, Gabsdil, Mylet, Popa and Shao~\cite{DGMPSh}, where
spreadness is used to control how often individual vertices are visited by
random walks, and hence to control possible chords. To our knowledge,
ours is the first argument to place a spread probability weight directly on
the cycles and use it to enforce inducedness.

We say that an edge set $J\subseteq E(K_M)$ \emph{supports} $\mu$ if
$E(Q)\subseteq J$ for every $Q\in\mathcal C_h$ with $\mu(Q)>0$. A graph
$G\subseteq K_M$ supports $\mu$ if $E(G)$ supports $\mu$. For a $k$-edge-colouring of $K_M$ and a colour $c$, let $G_c$ denote the
spanning subgraph of $K_M$ whose edges are precisely the edges of colour $c$.

We first show that the short monochromatic cycles admit a spread probability weight after fixing a suitable colour and cycle length.

\begin{lemma}\label{lem:spread-exists}
In every $k$-edge-colouring of $K_M$, there exist a colour $c$ and an odd $h$ with $5\le h\le4\ceil{\log_2 k}$ such that  $G_c$ supports a spread probability weight on $\mathcal C_h$.
\end{lemma}

\begin{proof}
Choose a $B$-vertex set $U\subseteq V(K_M)$ uniformly at random. By \cref{lem:short-cycle}, $K_M[U]$ contains a monochromatic odd cycle of length between $5$ and $4\lceil\log_2 k\rceil$; for each $U$, fix one such cycle $Q$.

There are at most $2k\lceil\log_2 k\rceil$ possible pairs consisting of the colour and length of $Q$. Hence, for some colour $c$ and some odd integer $h$ with $5\le h\le4\lceil\log_2 k\rceil$, the event 
$
\mathcal E=\{Q\text{ has colour }c\text{ and length }h\}
$
satisfies $\PP(\mathcal E)\ge1/(2k\lceil\log_2 k\rceil)$. Define $\mu(R)=\PP(Q=R\mid\mathcal E)$ for every $R\in\mathcal C_h$. Then $\mu$ is a probability weight supported by $G_c$.

We verify that $\mu$ is spread. Fix $S\subseteq V(K_M)$ with $2\le|S|\le h$. Since $V(Q)\subseteq U$,
\[
\begin{aligned}
\sum_{R\in\mathcal C_h:\,S\subseteq V(R)}\mu(R)
&=\PP(S\subseteq V(Q)\mid\mathcal E)
\le\frac{\PP(S\subseteq U)}{\PP(\mathcal E)}
\le
2k\lceil\log_2 k\rceil
\left(\frac BM\right)^{|S|}.
\end{aligned}
\]
Thus $\mu$ is spread, with the stronger constant $2$ in place of $8$.
\end{proof}

We next formulate a robust covering statement corresponding to the $k$ colour classes in a $k$-edge-colouring. Namely, we consider $k$ arbitrary edge sets and show that if none supports a spread probability weight, then their union leaves many edges of $K_M$ uncovered.
\begin{lemma}\label{lem:many-uncovered}
Let $X_1,\ldots,X_k\subseteq E(K_M)$. If none of $X_1,\ldots,X_k$
supports a spread probability weight on $\mathcal C_h$ for any odd $h$ with
$5\le h\le4\ceil{\log_2 k}$, then \vspace{-0.1cm}
\[
\left|E(K_M)\setminus\bigcup_{i=1}^k X_i\right|
\ge B^{-3}\binom M2.
\]
\end{lemma}

\begin{proof}
Let
$
Z=E(K_M)\setminus\bigcup_{i=1}^k X_i,
$
and suppose that $|Z|<B^{-3}\binom M2$. For each edge in
$\bigcup_{i=1}^k X_i$, choose arbitrarily an index $i$ such that the edge
belongs to $X_i$ and give it colour $i$. Colour the edges of $Z$
arbitrarily.

Then \cref{lem:spread-exists}  gives a colour $c$, an odd $h$ with
$5\le h\le4\ceil{\log_2 k}$, and a probability weight $\mu$ supported by
$G_c$ such that, for every $S\subseteq V(K_M)$ with
$2\le|S|\le h$, \vspace{-0.1cm}
\begin{equation}\label{eq:strong-spread}
\sum_{Q\in\mathcal C_h:\,S\subseteq V(Q)}\mu(Q)
\le
2k\ceil{\log_2 k}\left(\frac BM\right)^{|S|}.
\end{equation}
In particular, for every $uv\in E(K_M)$,\vspace{-0.1cm}
\[
\sum_{Q\in\mathcal C_h:\,uv\in E(Q)}\mu(Q)
\le
2k\ceil{\log_2 k}\left(\frac BM\right)^2.
\]

Delete all positive-weights of cycles containing an edge of $Z$. By the above
inequality and the bound on $|Z|$, a direct calculation shows that the
deleted weight is less than $1/10$.

Let $\widetilde\mu$ be the renormalisation of the remaining weight, i.e.,
obtained by dividing each remaining weight by the total, so
that $\sum_{Q\in\mathcal C_h}\widetilde\mu(Q)=1$. Then every positive-weight cycle under $\widetilde\mu$ has its edge set contained in $X_c$. Moreover,
the renormalising factor is at most $10/9$, so \eqref{eq:strong-spread} gives
\[
\sum_{Q\in\mathcal C_h:\,S\subseteq V(Q)}\widetilde\mu(Q)
\le
\frac{20}{9}k\ceil{\log_2 k}
\left(\frac BM\right)^{|S|}
\le
8k\ceil{\log_2 k}
\left(\frac BM\right)^{|S|}.
\]
Thus $X_c$ supports the spread probability weight $\widetilde\mu$,
a contradiction.
\end{proof}

\subsection{Overlap estimates for spread weights}\label{sec:overlap}
We next derive several consequences of the spread condition that control
overlaps between weighted cycles. 
Recall that $h$ is odd with $5\le h\le4\ceil{\log_2 k}$, and that
$\mathcal C_h$ denotes the family of all $h$-cycles in $K_M$.

\begin{lemma}\label{lem:overlap}
Let $\mu$ be a spread weight on $\mathcal C_h$.
\begin{enumerate}[label=\textup{(\roman*)},leftmargin=2.6em]
\item If $A\subseteq E(K_M)$ and $1\le|A|<h$, then
\[
\sum_{Q\in\mathcal C_h:\,A\subseteq E(Q)}\mu(Q)
\le
8k\lceil\log_2 k\rceil
\left(\frac BM\right)^{|A|+1}.
\]
If $|A|=h$, the corresponding sum is at most
$8k\lceil\log_2 k\rceil(B/M)^h$.

\item For $2\le j<h$,
\[
\sum_{\substack{A\subseteq E(K_M)\\|A|=j}}
\left(\sum_{Q\in\mathcal C_h:\,A\subseteq E(Q)}\mu(Q)\right)^2
\le
8k\lceil\log_2 k\rceil\binom hj
\left(\frac BM\right)^{j+1}.
\]
For $j=h$, the corresponding sum is at most
$8k\lceil\log_2 k\rceil(B/M)^h$.

\item
\[
\sum_{\substack{S\subseteq V(K_M)\\|S|=3}}
\left(\sum_{Q\in\mathcal C_h:\,S\subseteq V(Q)}\mu(Q)\right)^2
\le
8k\lceil\log_2 k\rceil\binom h3
\left(\frac BM\right)^3.
\]
\end{enumerate}
\end{lemma}

\begin{proof}
\textup{(i)} Let $\emptyset\ne A\subseteq E(K_M)$, and suppose that $A\subseteq E(Q_0)$ for some $h$-cycle $Q_0$, as otherwise the claim is trivial.

First assume that $|A|<h$, and let $S$ be the set of vertices incident with the edges of $A$. Since $A$ is a proper subset of $E(Q_0)$, we have $|S|\ge|A|+1$. Thus, by \eqref{eq:spread},
\[
\sum_{Q\in\mathcal C_h:\,A\subseteq E(Q)}\mu(Q)
\le
\sum_{Q\in\mathcal C_h:\,S\subseteq V(Q)}\mu(Q)
\le
8k\lceil\log_2 k\rceil
\left(\frac BM\right)^{|S|}
\le
8k\lceil\log_2 k\rceil
\left(\frac BM\right)^{|A|+1}.
\]

If $|A|=h$, then $A=E(Q_0)$, and hence \eqref{eq:spread},
\[
\sum_{Q\in\mathcal C_h:\,A\subseteq E(Q)}\mu(Q)
=\mu(Q_0)
\le
8k\lceil\log_2 k\rceil
\left(\frac BM\right)^h.
\]

\noindent{\textup{(ii)}} Fix $2\le j\le h$, and for each $A\subseteq E(K_M)$ with $|A|=j$, write
$
x_A=\sum_{Q\in\mathcal C_h:\,A\subseteq E(Q)}\mu(Q).
$
If $j<h$, then by part~\textup{(i)},
\[
\sum_{|A|=j}x_A^2
\le
\max_{|A|=j}x_A\cdot\sum_{|A|=j}x_A
\le
8k\lceil\log_2 k\rceil
\left(\frac BM\right)^{j+1}
\sum_{Q\in\mathcal C_h}\mu(Q)\binom hj
\le
8k\lceil\log_2 k\rceil\binom hj
\left(\frac BM\right)^{j+1}.
\]

If $j=h$, the same argument using part~\textup{(i)} gives the desired bound.

\smallskip 

\noindent{\textup{(iii)}} For each $S\subseteq V(K_M)$ with $|S|=3$, write
$
x_S=\sum_{Q\in\mathcal C_h:\,S\subseteq V(Q)}\mu(Q).
$
Then, by \eqref{eq:spread},
\[
\sum_{|S|=3}x_S^2
\le
\left(\max_{|S|=3}x_S\right)\sum_{|S|=3}x_S
\le
8k\lceil\log_2 k\rceil
\left(\frac BM\right)^3
\sum_{Q\in\mathcal C_h}\mu(Q)\binom h3
\le
8k\lceil\log_2 k\rceil\binom h3
\left(\frac BM\right)^3.
\]
\end{proof}

In the container argument later, the relevant family of edge sets is controlled
not by its size but by its total $p$-weight, $\sum_{A}p^{|A|}$. The next
lemma shows that such a bound implies that the family can capture only a
small amount of a spread weight.
\begin{lemma}\label{lem:p-weight-deletion}
Suppose that $M\ge B^{20}$ and
$B^{-9}\le p\le1$. Let $\mu$ be a spread weight on $\mathcal C_h$.
Suppose that $\mathcal A$ is a family of edge sets such that
$2\le|A|\le h$ for every $A\in\mathcal A$ and
\[
\sum_{A\in\mathcal A}p^{|A|}
\le
p\binom M2.
\]
Then the total $\mu$-weight of cycles whose edge set contains some member
of $\mathcal A$ is at most $1/100$.
\end{lemma}\begin{proof}
Fix $A\in\mathcal A$ and write $a=|A|$. By \cref{lem:overlap}(i), a direct calculation gives
\[
\sum_{Q\in\mathcal C_h:\,A\subseteq E(Q)}\mu(Q)
\le
\frac1{100}\frac{p^{a-1}}{\binom M2}.
\]
Here we use $a\ge2$ when $a<h$, while the case $a=h$ follows using $h\ge5$.
By the union bound, the total weight of cycles containing some member of
$\mathcal A$ is therefore at most
\[
\frac1{100p\binom M2}
\sum_{A\in\mathcal A}p^{|A|}
\le
\frac1{100}.
\]
\end{proof}

\vspace{-0.4cm}
\subsection{Chord scores and induced cycles}\label{sec:chords}
We now use the overlap estimates above to control the possible chords of
the weighted cycles and show that a sufficiently small chord score forces
an induced cycle. For an $h$-cycle $Q$, let $\operatorname{Ch}(Q)=\binom{V(Q)}2\setminus E(Q)$,
whose elements we call the \emph{possible chords} of $Q$. For
$xy\in E(K_M)$, define
$
c_\mu(xy)
=
\sum_{Q\in\mathcal C_h:\,xy\in\operatorname{Ch}(Q)}\mu(Q),
$
and set $\operatorname{ch}(\mu)=\sum_{xy\in E(K_M)}c_\mu(xy)^2$.
We call $\operatorname{ch}(\mu)$ the \emph{chord score} of $\mu$. A small
chord score means that the possible chords are well distributed.

\begin{lemma}\label{lem:spread-chord}
Let $\mu$ be a spread weight on $\mathcal C_h$. Then
\[
\operatorname{ch}(\mu)
\le
8k\ceil{\log_2 k}h^2
\left(\frac BM\right)^2.
\]
\end{lemma}

\begin{proof}
By \eqref{eq:spread},
$c_\mu(xy)\le8k\ceil{\log_2 k}(B/M)^2$ for every $xy\in E(K_M)$.
Moreover,
$\sum_{xy}c_\mu(xy)\le h^2\sum_{Q\in\mathcal C_h}\mu(Q)\le h^2$.
Hence
\[
\operatorname{ch}(\mu)
\le
\left(\max_{xy}c_\mu(xy)\right)\sum_{xy}c_\mu(xy)
\le
8k\ceil{\log_2 k}h^2
\left(\frac BM\right)^2.
\]
\end{proof}

To understand how the chord score behaves after randomly sparsifying $E(K_M)$, we separate
the contribution coming from pairs of cycles according to their overlap.
For $0\le j\le h$, define
\[
\operatorname{ch}_j(\mu)
=
\sum_{xy\in E(K_M)}
\sum_{\substack{Q,Q'\in\mathcal C_h:\\
xy\in \operatorname{Ch}(Q)\cap\operatorname{Ch}(Q')\\
|E(Q)\cap E(Q')|=j}}
\mu(Q)\mu(Q').
\] 
Thus
$
\operatorname{ch}(\mu)
=
\sum_{j=0}^h\operatorname{ch}_j(\mu).
$
\begin{lemma}\label{lem:chord-overlap}
Let $\mu$ be a spread weight on $\mathcal C_h$. Then
\begin{enumerate}[label=\textup{(\roman*)},leftmargin=2.6em]
\item
\[
\operatorname{ch}_1(\mu)
\le
8k\ceil{\log_2 k}h\binom h3
\left(\frac BM\right)^3.
\]

\item For $2\le j<h$,
\[
\operatorname{ch}_j(\mu)
\le
8k\ceil{\log_2 k}h^2\binom hj
\left(\frac BM\right)^{j+1}.
\]

\item
\[
\operatorname{ch}_h(\mu)
\le
8k\ceil{\log_2 k}h^2
\left(\frac BM\right)^h.
\]
\end{enumerate}
\end{lemma}

\begin{proof}
\noindent\textup{(i)}
Fix $Q,Q'\in\mathcal C_h$ with $E(Q)\cap E(Q')=\{uv\}$ and a common
possible chord $xy$. Since $xy\ne uv$, at least one endpoint of $xy$ lies outside $\{u,v\}$;
choose one such endpoint and call it $z$. Then
$\{u,v,z\}\subseteq V(Q)\cap V(Q')$. For each fixed $Q,Q'$ and chosen
triple $\{u,v,z\}$, the common possible chord has the form $zw$ with
$w\in V(Q)\cap V(Q')$, so there are at most $h$ choices for $w$.
Therefore
\[
\begin{aligned}
\operatorname{ch}_1(\mu)
&=
\sum_{xy}
\sum_{\substack{Q,Q'\in\mathcal C_h:\\
xy\in \operatorname{Ch}(Q)\cap\operatorname{Ch}(Q'),\\
|E(Q)\cap E(Q')|=1}}
\mu(Q)\mu(Q')\\
&\le
h
\sum_{\substack{S\subseteq V(K_M)\\ |S|=3}}
\sum_{\substack{Q,Q'\in\mathcal C_h:\\
S\subseteq V(Q)\cap V(Q')}}
\mu(Q)\mu(Q')
\qquad
\text{(grouping them according to $S=\{u,v,z\}$)}\\
&=
h
\sum_{\substack{S\subseteq V(K_M)\\ |S|=3}}
\left(
\sum_{Q\in\mathcal C_h:\,S\subseteq V(Q)}
\mu(Q)
\right)^2
\\
&\le
8k\ceil{\log_2 k}\,h\binom h3
\left(\frac BM\right)^3
\qquad\qquad\qquad
\text{(by \cref{lem:overlap}(iii)).}
\end{aligned}
\]

\noindent\textup{(ii)}
Let $2\le j<h$. Since two $h$-cycles have fewer than $h^2$ common
possible chords, \vspace{-0.1cm}
\[
\begin{aligned}
\operatorname{ch}_j(\mu)
&\le
h^2
\sum_{\substack{A\subseteq E(K_M)\\ |A|=j}}
\left(
\sum_{Q\in\mathcal C_h:\,A\subseteq E(Q)}\mu(Q)
\right)^2
\quad
\text{(summing over the common $j$-edge set $A$)}\\
&\le
8k\ceil{\log_2 k}\,h^2\binom hj
\left(\frac BM\right)^{j+1}
\qquad 
\quad \ \
\text{(by \cref{lem:overlap}(ii)).}
\end{aligned}
\]

\noindent\textup{(iii)}
If $j=h$, then $Q=Q'$. Hence
\[
\begin{aligned}
\operatorname{ch}_h(\mu)
&\le
h^2\sum_{Q\in\mathcal C_h}\mu(Q)^2
\\
&\le
h^2
\left(\max_{Q\in\mathcal C_h}\mu(Q)\right)
\sum_{Q\in\mathcal C_h}\mu(Q)
\\
&\le
8k\ceil{\log_2 k}\,h^2
\left(\frac BM\right)^h
\qquad\quad \quad 
\text{(by \eqref{eq:spread} and $\sum_{Q\in \mathcal C_h}\mu(Q)\le1$).}
\end{aligned}
\]
\end{proof}

The chord score provides a simple criterion for finding an induced
cycle. If a graph supporting $\mu$ has sufficiently few edges relative to
$\operatorname{ch}(\mu)$, then some cycle is induced.

\begin{lemma}\label{lem:chord-test}
Let $G\subseteq K_M$ support a probability weight $\mu$ on $\mathcal C_h$.
If $e(G)\operatorname{ch}(\mu)<1$, then some $Q\in\mathcal C_h$ with
$\mu(Q)>0$ is induced in $G$.
\end{lemma}

\begin{proof}
Suppose that every positive-weight $h$-cycle has a chord in $G$. Choose an $h$-cycle according to $\mu$. Its expected number of chords in $G$ is at least one, while by linearity of expectation and Cauchy--Schwarz, \vspace{-0.1cm}
\[
1\le
\sum_{xy\in E(G)}c_\mu(xy)
\le
\sqrt{e(G)\sum_{xy}c_\mu(xy)^2}
=
\sqrt{e(G)\operatorname{ch}(\mu)}
<1,
\]
a contradiction.
\end{proof}

\subsection{Conditional thinning with controlled chord score}\label{sec:conditional-thinning}
We show that a spread weight on short odd cycles can be thinned, by retaining only those cycles contained in a suitable random edge set, even after conditioning on a decreasing event, while keeping enough total weight to renormalise the surviving weights without losing control of the chord score.

 An event
$\mathcal D\subseteq2^{E(K_M)}$ is \emph{decreasing} if $Y\in\mathcal D$
and $Y'\subseteq Y$ imply $Y'\in\mathcal D$. 

\begin{lemma}\label{lem:conditional-thinning}
Let $k\ge20$, set $B=16^k$, $M=B^{20}$, and $p=B^{-9}$, and let $h$ be odd
with $5\le h\le4\ceil{\log_2 k}$. Let $\mu_0$ be a spread weight on
$\mathcal C_h$ satisfying
$
\sum_{Q\in\mathcal C_h}\mu_0(Q)\ge0.99.
$
Let $\mathcal D\subseteq2^{E(K_M)}$ be a decreasing event of positive
probability, and let $Y$ be a $p$-random subset of $E(K_M)$. Suppose that
\[
\PP(E(Q)\subseteq Y\mid\mathcal D)\ge(p/2)^h
\qquad
\text{for every $Q\in\mathcal C_h$ with $\mu_0(Q)>0$.}
\]
Then some $Y\in\mathcal D$ supports a probability weight on $\mathcal C_h$
with chord score at most $B^{-37}$.
\end{lemma}
\begin{proof}
For $Y\subseteq E(K_M)$, let
$
Z(Y)
=
\sum_{Q\in\mathcal C_h:\,E(Q)\subseteq Y}\mu_0(Q)
$
and \vspace{-0.1cm}
\[
W(Y)
=
\sum_{xy\in E(K_M)}
\left(
\sum_{\substack{Q\in\mathcal C_h:\,E(Q)\subseteq Y\\
xy\in\operatorname{Ch}(Q)}}
\mu_0(Q)
\right)^2.
\]
If $Z(Y)>0$, define a probability weight $\mu_Y$ on $\mathcal C_h$ as follows:
for each $h$-cycle $Q\in\mathcal C_h$, set
\[
\mu_Y(Q)
=
\begin{cases}
\mu_0(Q)/Z(Y), & \text{if } E(Q)\subseteq Y,\\
0, & \text{otherwise}.
\end{cases}
\]
Then $\mu_Y$ is a probability weight  supported by $Y$ on $\mathcal C_h$, and its chord
score is $W(Y)/Z(Y)^2$.
\begin{claim}\label{clm:ratio}
There exists $Y\in\mathcal D$ with $Z(Y)>0$ such that
\[
\frac{W(Y)}{Z(Y)^2}
\le
\frac{\E[W\mid\mathcal D]}{\E[Z\mid\mathcal D]^2}.
\]
\end{claim}

\begin{proof}
Otherwise, writing the right-hand side as $A$, we would have
$W(Y)>AZ(Y)^2$ whenever $Z(Y)>0$. Since $Z(Y)=0$ implies $W(Y)=0$,
taking conditional expectations gives
\[
\E[W\mid\mathcal D]
>
A\E[Z^2\mid\mathcal D]
\ge
A\E[Z\mid\mathcal D]^2
=
\E[W\mid\mathcal D],
\]
a contradiction.
\end{proof}

\begin{claim}\label{clm:EW}
We have \vspace{-0.1cm}
\[
\E[Z\mid\mathcal D]\ge0.99(p/2)^h
\qquad\text{and}\qquad
\E[W\mid\mathcal D]
\le
32k\ceil{\log_2 k}^5p^{2h}B^{-38}.
\]
\end{claim}

\begin{proof}
The first inequality follows immediately from the assumption of the lemma.
For the second, let $Q,Q'\in\mathcal C_h$ and write
$j=|E(Q)\cap E(Q')|$. By Harris's
inequality~\cite{Harris}, an increasing event and a decreasing event are
negatively correlated. Since $\{Y : E(Q)\cup E(Q')\subseteq Y\}$ is an increasing event
and $\mathcal D$ is decreasing,
\begin{equation}\label{eq:pair-survival}
\PP(E(Q)\cup E(Q')\subseteq Y\mid\mathcal D)
\le
\PP(E(Q)\cup E(Q')\subseteq Y)
=
p^{2h-j}.
\end{equation}

Expanding the square in the definition of $W$ and taking conditional
expectation, we obtain
\[
\begin{aligned}
\E[W\mid\mathcal D]
&=
\sum_{xy\in E(K_M)}
\sum_{\substack{Q,Q'\in\mathcal C_h\\
xy\in\operatorname{Ch}(Q)\cap\operatorname{Ch}(Q')}}
\mu_0(Q)\mu_0(Q')\,
\PP(E(Q)\cup E(Q')\subseteq Y\mid\mathcal D)\\
&\le
\sum_{j=0}^{h}
p^{2h-j}
\sum_{xy\in E(K_M)}
\sum_{\substack{Q,Q'\in\mathcal C_h\\
xy\in\operatorname{Ch}(Q)\cap\operatorname{Ch}(Q')\\
|E(Q)\cap E(Q')|=j}}
\mu_0(Q)\mu_0(Q')\\
&\le
p^{2h}\operatorname{ch}(\mu_0)
+
p^{2h-1}\operatorname{ch}_1(\mu_0)
+
\sum_{j=2}^{h-1}p^{2h-j}\operatorname{ch}_j(\mu_0)
+
p^h\operatorname{ch}_h(\mu_0).
\end{aligned}
\]

By \cref{lem:spread-chord,lem:chord-overlap}, together with
$M=B^{20}$ and $p=B^{-9}$,
\[
\begin{aligned}
\E[W\mid\mathcal D]
&\le
8k\ceil{\log_2 k}p^{2h}
\left[
h^2B^{-38}
+
h\binom h3 B^{-48}
+
h^2\sum_{j=2}^{h-1}\binom hj B^{-10j-19}
+
h^2B^{-10h}
\right].
\end{aligned}
\]
Since $h\le4\ceil{\log_2 k}$, a direct calculation shows that each of the
four terms in the brackets is at most
$\ceil{\log_2 k}^4B^{-38}$. Hence
$
\E[W\mid\mathcal D]
\le
32k\ceil{\log_2 k}^5p^{2h}B^{-38}.
$
\end{proof}
For the $Y$ given by Claim~\ref{clm:ratio}, Claim~\ref{clm:EW} gives
$
{W(Y)}/{Z(Y)^2}
\le
\frac{32}{0.99^2}k\ceil{\log_2 k}^5 4^h B^{-38}
\le
B^{-37},
$
where the last inequality follows from $k\ge20$ and
$h\le4\ceil{\log_2 k}$.
Thus the probability weight $\mu_Y$ has chord score at most $B^{-37}$.
\end{proof}

\section{Compression into chord-light containers}\label{sec:compression}
We now isolate the edge sets that are already sufficient for the final
induced-cycle argument. Throughout this section, let $k$ be sufficiently large, and let $B=16^k$, $M=B^{20}$, and $p=B^{-9}$. By
\cref{lem:chord-test}, if a $k$-coloured graph with at most $2B^{36}$ edges
has a colour class supporting a probability weight with chord score at most
$B^{-37}$, then it contains a monochromatic induced cycle. This motivates
the following definition.

 An edge set
$I\subseteq E(K_M)$ is \emph{chord-light} if, for some odd $h$ with
$5\le h\le4\ceil{\log_2 k}$, it supports a probability weight on
$\mathcal C_h$ with chord score at most $B^{-37}$. Observe that since $M=B^{20}$, \cref{lem:spread-chord} implies that every edge set
supporting a spread probability weight is chord-light.

The following compression lemma is the key structural reduction behind the
final construction. It reduces the entire class of non-chord-light edge sets
to a family of at most $\exp(B^{32})$ containers, each of which does not support a spread probability weight on
$\mathcal C_h$ for any odd $h$ with
$5\le h\le4\ceil{\log_2 k}$.
Thus, the failure of chord-lightness admits a bounded-complexity description
in terms of containers carrying no spread odd-cycle structure.

\begin{lemma}\label{lem:compression}
There is a family $\mathfrak C\subseteq 2^{E(K_M)}$ such that:
\begin{enumerate}[label=\textup{(\roman*)},leftmargin=2.4em]
\item every non-chord-light edge set $I\subseteq E(K_M)$ is contained in some
$X\in\mathfrak C$;
\item no $X\in\mathfrak C$ supports a spread probability weight on
$\mathcal C_h$ for any odd $h$ with
$5\le h\le4\ceil{\log_2 k}$;
\item $|\mathfrak C|\le\exp(B^{32})$.
\end{enumerate}
\end{lemma}
We briefly describe the proof of Lemma~\ref{lem:compression}. We regard
$E(K_M)$ as the vertex set of an auxiliary hypergraph. Let
$\mathcal B\subseteq 2^{E(K_M)}$ consist of all chord-light edge sets $J$
such that no proper subset of $J$ is chord-light. Thus the members of
$\mathcal B$ are the minimal obstructions to being non-chord-light, and
\begin{equation}\label{eq:minimal-chord-light}
I\text{ is not chord-light}
\quad\Longleftrightarrow\quad
I\text{ is }\mathcal B\text{-free}.
\end{equation}

Starting from a $\mathcal B$-free edge set $I\subseteq E(K_M)$, we first use
a container argument to encode the obstruction represented by $I$ using a
small fingerprint. This reduces the original non-uniform problem, in which the members of
$\mathcal B$ may have different sizes, to a collection of $h$-uniform
problems, one for each relevant cycle length $h$. 
A second container argument can then be applied separately at each length,
placing $I$ inside a container $X_h$ determined by another small
fingerprint.

The key step is to show that these containers already exclude the structure
we ultimately want to force: no $X_h$ can support a spread probability
weight on $\mathcal C_h$. Conceptually, the first container controls the
forbidden chord-light structure, while the second isolates the relevant
$h$-cycles inside a bounded-complexity set. The conditional thinning lemma, \cref{lem:conditional-thinning}, then
connects the two: if a spread weight survived inside $X_h$, it could be
thinned to produce a chord-light edge set that is still $\mathcal B$-free,
contradicting \eqref{eq:minimal-chord-light}.

Finally, we intersect the containers $X_h$ over all relevant odd lengths.
Each non-chord-light edge set is contained in one such intersection, while
the intersection supports no spread probability weight on any relevant
family of odd cycles. Since every container is determined by a small
collection of fingerprints, there are only few possible intersections.
A direct counting argument then gives
$|\mathfrak C|\le\exp(B^{32})$.

\subsection{Container preliminaries}
The hypergraph container method was developed independently by Balogh,
Morris and Samotij~\cite{BMS} and Saxton and Thomason~\cite{ST}.  We use two recent refinements, stated directly in the edge-set language.

Let $\mathcal F\subseteq 2^{E(K_M)}$ be a family of edge sets. An edge set
$I\subseteq E(K_M)$ is \emph{$\mathcal F$-free} if it contains no member
of $\mathcal F$.  For a family $\mathcal A\subseteq 2^{E(K_M)}$, we say that $\mathcal A$
\emph{covers} $\mathcal F$ if every member of $\mathcal F$ contains some
member of $\mathcal A$. For $T\subseteq E(K_M)$, let
$\underline{\partial}_T\mathcal F=\{F\setminus T:F\in\mathcal F\}$.
In our application, $\mathcal F$ will be the family $\mathcal B$ of minimal
chord-light edge sets.

The first lemma follows from Theorem~4.3 and Observation~4.5 of
Arag\~ao, Campos, Dahia, Filipe and Marciano~\cite{ACDFM}, with
$\alpha=1/2$. It applies to an arbitrary forbidden family $\mathcal F$ and
associates with each $\mathcal F$-free edge set $I$ a small fingerprint
$T\subseteq I$ and a cover $\mathcal C_T$ of $\mathcal F$. It also gives a lower bound on the conditional survival probability of every
$A\notin\mathcal C_T$ under $p$-random sampling; without conditioning, this
probability is simply $p^{|A|}$.

\begin{lemma}[\cite{ACDFM}]
\label{lem:conditional-container}
Let $M\ge2$ and $0<p\le1/2$, and let $\mathcal F\subseteq 2^{E(K_M)}$. Then every $\mathcal F$-free edge set $I$ has a fingerprint
$T\subseteq I$ with $|T|\le2p\binom M2$ and an associated family
$\mathcal C_T\subseteq 2^{E(K_M)}$ such that:

\begin{enumerate}[label=\textup{(\roman*)},leftmargin=2.4em]
\vspace{-0.1cm}
\item $\mathcal C_T$ covers $\mathcal F$;
\vspace{-0.1cm}
\item $I$ is $\mathcal C_T$-free;
\vspace{-0.1cm}
\item Let $Y$ be a $p$-random subset of $E(K_M)$. Then \vspace{-0.1cm}
$$\PP(A\subseteq Y\mid Y\text{ is }\underline{\partial}_T\mathcal F
\text{-free})>(p/2)^{|A|} \quad \text{for every nonempty set
$A\notin\mathcal C_T$}.
$$
\end{enumerate}
\end{lemma}

In our application, we take $\mathcal F=\mathcal B$. Thus, for each $\mathcal B$-free (equivalently, non-chord-light) edge set
$I$, we obtain a fingerprint $T$ and the associated cover $\mathcal C_T$. For an odd integer $h$ with
$5\le h\le4\ceil{\log_2 k}$, we then take
$\mathcal H\subseteq\binom{E(K_M)}{h}$ to be the family of edge sets of
$h$-cycles containing some member of $\mathcal C_T$. Thus every $\mathcal C_T$-free edge set is $\mathcal H$-free, allowing us
to apply the $h$-uniform container lemma.
In summary, for the fingerprint $T$ associated with $I$, we have
\[
I\text{ is not chord-light}
\quad\Longleftrightarrow\quad
I\text{ is }\mathcal B\text{-free}
\quad\Longrightarrow\quad
I\text{ is }\mathcal C_T\text{-free}
\quad\Longrightarrow\quad
I\text{ is }\mathcal H\text{-free}.
\]

We will also use Theorem~A of Campos and Samotij~\cite{CS}. It places every
$\mathcal H$-free edge set inside a container $X$. It also gives a family $\mathcal G\subseteq 2^{X}$ controlling the members of $\mathcal H$ contained in
$X$, with total $p$-weight at most $p|X|$.

\begin{lemma}[\cite{CS}]\label{lem:uniform-container}
Let $r\ge2$, $\mathcal H\subseteq\binom{E(K_M)}{r}$, and let
$I\subseteq E(K_M)$ be $\mathcal H$-free. Suppose that
$0<p\le1/(8r^2)$. Then there is a fingerprint $S\subseteq I$ with
$|S|\le8r^2p\binom M2$ and a container $X_S\subseteq E(K_M)$, such that $I\subseteq X_S$. Moreover, there is a family $\mathcal G\subseteq 2^{X_S}$ such that:
\begin{enumerate}[label=\textup{(\roman*)},leftmargin=2.4em]\vspace{-0.1cm}
\item every $F\in\mathcal H$ with $F\subseteq X_S$ contains some
$A\in\mathcal G$;

\vspace{-0.1cm}\item every $A\in\mathcal G$ has at least two edges;

\vspace{-0.1cm}\item $\sum_{A\in\mathcal G}p^{|A|}\le p|X_S|$.
\end{enumerate}
\end{lemma}
In our application, we take $r=h$. Since $I$ is
$\mathcal H$-free, the lemma gives a container $X_h$ with $I\subseteq X_h$. Every $F\in\mathcal H$ with $F\subseteq X_h$ contains some
$A\in\mathcal G_h$. Discard from $\mathcal G_h$ every edge set which is not
contained in any member of $\mathcal H$. The remaining family still covers
all members of $\mathcal H$ contained in $X_h$, and every
$A\in\mathcal G_h$ satisfies $2\le |A|\le h$. Moreover, \vspace{-0.1cm}
\[\vspace{-0.1cm}
\sum_{A\in\mathcal G_h}p^{|A|}
\le p|X_h|
\le p\binom M2.
\]
By \cref{lem:p-weight-deletion}, the cycles containing some member of
$\mathcal G_h$ have total spread weight at most $1/100$. Thus, from any
spread probability weight supported on $X_h$, we may delete at most
$1/100$ of its total weight so that every remaining positive-weight cycle
$Q$ lies outside $\mathcal H$, and hence contains no member of
$\mathcal C_T$.
 In particular, $E(Q)\notin\mathcal C_T$, so
\cref{lem:conditional-container}, applied with $A=E(Q)$, gives
\[
\PP(E(Q)\subseteq Y\mid
Y\text{ is }\underline\partial_T\mathcal B\text{-free})
>
(p/2)^h.
\]
This is the lower bound required in \cref{lem:conditional-thinning}.
\subsection{The compression argument}

The setup above associates to each $\mathcal B$-free edge set $I$, a
fingerprint $T$ and cover $\mathcal C_T$, and then, for each relevant odd
$h$, a second fingerprint $S_h$, a container $X_h$, and a family
$\mathcal G_h$. We show that $X_h$ cannot support a spread
probability weight. 

\begin{lemma}\label{lem:spread-exclusion}
Let $I$ be a $\mathcal B$-free edge set, and let $T$ and $\mathcal C_T$ be
supplied by \cref{lem:conditional-container}. For an odd $h$ with
$5\le h\le4\ceil{\log_2 k}$, let $\mathcal H_{T,h}$ be the family of edge
sets of $h$-cycles containing some member of $\mathcal C_T$, and let
$S_h$, $X_h$, and $\mathcal G_h$ be supplied by
\cref{lem:uniform-container}. Then $X_h$ supports no spread probability
weight on $\mathcal C_h$.
\end{lemma}

\begin{proof}
Suppose that $X_h$ supports  a spread probability weight $\mu$ on
$\mathcal C_h$.

\begin{claim}\label{clm:prune-spread}
There is a spread weight $\mu_0$ on $\mathcal C_h$ such that
$
\sum_{Q\in\mathcal C_h}\mu_0(Q)\ge0.99,
$
and every $Q$ with $\mu_0(Q)>0$ satisfies
$E(Q)\notin\mathcal H_{T,h}$.
\end{claim}

\begin{proof}
Discard from $\mathcal G_h$ every edge set $A\in\mathcal G_h$ for which there
is no cycle $Q$ with  $E(Q)\in\mathcal H_{T,h}$ such that
$A\subseteq E(Q)$. The remaining family still satisfies \cref{lem:uniform-container}(i), and
every remaining $A$ satisfies $2\le |A|\le h$. By
\cref{lem:uniform-container}(iii),
\[
\sum_{A\in\mathcal G_h}p^{|A|}
\le
p|X_h|
\le
p\binom M2.
\]
By \cref{lem:p-weight-deletion}, we may restrict $\mu$ to a spread weight
$\mu_0$ of total weight at least $0.99$ such that no positive-weight cycle
contains a member of $\mathcal G_h$. Since $\mu$ is supported on $X_h$,
every $Q$ with $\mu_0(Q)>0$ satisfies $E(Q)\subseteq X_h$. Hence
$E(Q)\notin\mathcal H_{T,h}$ by \cref{lem:uniform-container}(i), and
therefore $E(Q)$ contains no member of $\mathcal C_T$.
\end{proof}

Let $\mathcal D$ be the event that a $p$-random edge set
$Y\subseteq E(K_M)$ is $\underline{\partial}_T\mathcal B$-free.
Equivalently, $Y\in\mathcal D$ if and only if
$T\cup Y$ is $\mathcal B$-free.
Since $T\subseteq I$ and $I$ is $\mathcal B$-free,
the event $\mathcal D$ has positive probability.
Moreover, every $Y\in\mathcal D$ is $\mathcal B$-free
and hence non-chord-light.

For every $Q$ with $\mu_0(Q)>0$, we have
$E(Q)\notin\mathcal H_{T,h}$, and hence
$E(Q)\notin\mathcal C_T$. Thus,
\cref{lem:conditional-container} gives
$
\PP(E(Q)\subseteq Y\mid\mathcal D)>(p/2)^h.
$
By \cref{lem:conditional-thinning}, there exists
$Y\in\mathcal D$ which is chord-light.
This contradicts the fact that every $Y\in\mathcal D$
is non-chord-light.
\end{proof}

\begin{proof}[Proof of \cref{lem:compression}]
For every possible first fingerprint $T$ and every tuple of second
fingerprints $(S_h)_h$ arising from the setup above, let $X_h$ be the
corresponding container and set \vspace{-0.1cm}
\[
X(T,(S_h)_h)
=
\bigcap_{\substack{5\le h\le4\ceil{\log_2 k}\\h\text{ odd}}}X_h.
\vspace{-0.2cm}\]
Let $\mathfrak C$ consist of all such intersections.

\begin{claim}\label{clm:compression-cover}
Every non-chord-light edge set is contained in some member of
$\mathfrak C$, and no member of $\mathfrak C$ supports a spread probability
weight on $\mathcal C_h$ for any odd $h$ with
$5\le h\le4\ceil{\log_2 k}$.
\end{claim}

\begin{proof}
Let $I$ be non-chord-light. By \eqref{eq:minimal-chord-light}, $I$ is
$\mathcal B$-free, so \cref{lem:conditional-container} gives a fingerprint
$T$ and cover $\mathcal C_T$. Since $I$ is $\mathcal C_T$-free, it is
$\mathcal H_{T,h}$-free for every relevant $h$. Applying
\cref{lem:uniform-container} for each $h$ gives fingerprints $S_h$ and
containers $X_h$ with $I\subseteq X_h$. Hence
\(
I\subseteq X(T,(S_h)_h).
\)

By \cref{lem:spread-exclusion}, each $X_h$ supports no spread probability
weight on $\mathcal C_h$. Since $X(T,(S_h)_h)\subseteq X_h$, the same holds
for every member of $\mathfrak C$.
\end{proof}

It remains to show \textup{(iii)}. The first fingerprint has at most
$2p\binom M2$ edges, while for each relevant $h$ the second fingerprint has
at most $8h^2p\binom M2$ edges. 
Using
$
\sum_{j\le uN}\binom Nj\le(e/u)^{uN}
$
for $u\in [0,1/2]$, we have
\[
\begin{aligned}
|\mathfrak C|
&\le
\left(\frac{e}{2p}\right)^{2p\binom M2}
\prod_{\substack{5\le h\le4\ceil{\log_2 k}\\ h\text{ odd}}}
\left(\frac{e}{8h^2p}\right)^{8h^2p\binom M2}\le
\exp\left(100p\binom M2\ceil{\log_2 k}^3\log(1/p)\right)
\le
\exp(B^{32}).
\end{aligned}
\vspace{-0.5cm}\]
\end{proof}
We have now established the main structural lemma needed for the proof of
\cref{thm:gadget}. We will also use the following simple greedy observation.\begin{lemma}\label{lem:greedy-hitting}
Let $\mathfrak C\subseteq 2^{E(K_M)}$. Suppose that \vspace{-0.1cm}
\[
\left|E(K_M)\setminus\bigcup_{i=1}^k X_i\right|
\ge
\delta\binom M2
\quad\text{for every }X_1,\ldots,X_k\in\mathfrak C.
\]
Then there is an edge set $F\subseteq E(K_M)$ of size at most
$\left({2k\log|\mathfrak C|}\right)/{\delta}$ such that \vspace{-0.1cm}
\[
F\not\subseteq\bigcup_{i=1}^k X_i
\quad\text{for every }X_1,\ldots,X_k\in\mathfrak C.
\vspace{-0.2cm}\]
\end{lemma}
\begin{proof}
There are at most $|\mathfrak C|^k$ sets of the form
$E(K_M)\setminus\bigcup_{i=1}^k X_i$, and each has size at least
$\delta\binom M2$. If $r$ of them remain disjoint from the chosen edges,
some edge belongs to at least $\delta r$ of them. Thus each new edge reduces
their number by a factor of at most $1-\delta$. After at most
$(k\log|\mathfrak C|+1)/\delta\le2k\log|\mathfrak C|/\delta$ choices, we have
$
F\not\subseteq\bigcup_{i=1}^k X_i
$
for every $X_1,\ldots,X_k\in\mathfrak C$.
\end{proof}

\section{Proof of \cref{thm:gadget}}\label{sec:gadget-proof}
It suffices to prove the result for sufficiently large $k$. Recall that
$B=16^k$ and $M=B^{20}$.
By \cref{lem:compression}, there is a family
$\mathfrak C\subseteq2^{E(K_M)}$ satisfying
\cref{lem:compression}(i)--(iii). By
\cref{lem:compression}(ii) and \cref{lem:many-uncovered}, \vspace{-0.1cm}
\[
\left|E(K_M)\setminus\bigcup_{i=1}^k X_i\right|
\ge
B^{-3}\binom M2
\qquad
\text{for every }X_1,\ldots,X_k\in\mathfrak C.
\]
Since $|\mathfrak C|\le\exp(B^{32})$, \cref{lem:greedy-hitting} gives a
subgraph $F\subseteq K_M$ with
$
e(F)\le2kB^{35}\le2B^{36}
$
such that \vspace{-0.2cm}
\[
E(F)\not\subseteq\bigcup_{i=1}^k X_i
\qquad
\text{for every }X_1,\ldots,X_k\in\mathfrak C.
\]

Now consider any $k$-edge-colouring of $F$, and let
$I_1,\ldots,I_k$ be its colour classes. We claim that some $I_i$ is
chord-light. Otherwise, by \cref{lem:compression}(i), for each $i$ there
is some $X_i\in\mathfrak C$ with $I_i\subseteq X_i$. Hence
$E(F)=\bigcup_{i=1}^k I_i
\subseteq
\bigcup_{i=1}^k X_i,
$
contrary to the choice of $F$.

Thus some colour class, say $I_c$, is chord-light. By definition, there are
an odd $h$ with $5\le h\le4\ceil{\log_2 k}$ and a probability weight
$\mu$ on $\mathcal C_h$, supported by $I_c$, such that
$\operatorname{ch}(\mu)\le B^{-37}$. Note that $
e(F)\operatorname{ch}(\mu)
\le
2B^{36}B^{-37}
<1.
$
By \cref{lem:chord-test}, some $Q\in\mathcal C_h$ with $\mu(Q)>0$ is
induced in $F$. Since $\mu$ is supported by $I_c$, the cycle $Q$ is
monochromatic. Finally,
$
e(F)\le2B^{36}=e^{O(k)},
$
which proves the theorem.
\qed

\section{Concluding remarks}\label{sec:concluding}
We conclude with several related problems and directions. 

\smallskip
\noindent\textbf{Length of the forced odd hole.}
Theorem~\ref{thm:gadget} shows that a graph with $e^{O(k)}$ edges can force
a monochromatic induced odd cycle of length $O(\log k)$. It is natural to
ask whether this logarithmic bound can be improved. A recent result of
Steiner~\cite{Steiner} shows that, for every fixed $p$, the multicolour
Ramsey number of $\{C_3,C_5,\ldots,C_{2p+1}\}$ is superexponential in $k$.
Consequently, no graph with $e^{O(k)}$ edges can force a monochromatic odd
cycle of bounded length, even without requiring the cycle to be induced.
Thus the length of the odd hole in Theorem~\ref{thm:gadget} must tend to
infinity with $k$.

\begin{question}
What is the smallest possible order of a function $f(k)$ for which there
exists an $e^{O(k)}$-edge graph whose every $k$-edge-colouring contains a
monochromatic induced odd cycle $C_h$ with $5\le h\le f(k)$?
In particular, is $f(k)=o(\log k)$ possible?
\end{question}

\smallskip
\noindent\textbf{Long induced subdivisions.} Javadi, Kohayakawa and Miralaei~\cite{JKM} recently studied the induced
size--Ramsey numbers of long subdivisions. They proved that, for every
$k,D\ge2$, if $H$ has maximum degree at most $D$ and $H^\sigma$ is an
$n$-vertex subdivision of $H$ in which every
subdivision path has length at least $c(k,D)\log_D n$, then
$
\widehat R_{\mathrm{ind}}(H^\sigma;k)
\le e^{O(k\log k)}D^9\log D\,n.
$

Our odd-hole gadget suggests a possible route to improving the dependence
on $k$. Indeed, an induced odd cycle $C_{2a+1}$ contains, between two
suitable vertices, two induced paths of lengths $a$ and $a+1$. Thus
Theorem~\ref{thm:gadget} provides the type of local length flexibility
needed in their subdivision framework, at a cost of only $e^{O(k)}$.
It seems plausible that incorporating our gadget into their argument could give
$
\widehat R_{\mathrm{ind}}(H^\sigma;k)
\le e^{O_D(k)}n
$.

\section*{Acknowledgements}
Lanchao Wang was supported by the National Key R\&D Program of China under
grant number 2024YFA1013900, the NSFC
under grant number 12471327,  the China Scholarship Council, and the
Institute for Basic Science (IBS-R029-C4).  Xiaolin Wang was supported by  the NSFC under grant number 12401447.  

The
authors used ChatGPT 5.5 and 5.6 as research partners for discussion, formula
derivations, literature search, and writing assistance. In the final manuscript, the idea of using spread distributions to control possible chords was developed by the authors. Many of the technical components, in particular the container-based compression argument, were developed with the assistance of ChatGPT. All mathematical arguments and proofs in the final
manuscript were written and verified by the authors.
\bibliographystyle{abbrv}
\bibliography{reference}
\end{document}